\documentclass{amsart}
\usepackage{graphicx}
\usepackage{amssymb}
\usepackage{epstopdf}
\usepackage{amsmath,amsthm,amscd,amssymb}
\usepackage{latexsym}
\usepackage[colorlinks,citecolor=red,pagebackref,hypertexnames=false]{hyperref}
\usepackage{geometry}                
\usepackage{color, xcolor}
\numberwithin{equation}{section}

\theoremstyle{plain}
\newtheorem{theorem}{Theorem}[section]
\newtheorem{lemma}[theorem]{Lemma}
\newtheorem{corollary}[theorem]{Corollary}
\newtheorem{proposition}[theorem]{Proposition}

\theoremstyle{definition}
\newtheorem{definition}[theorem]{Definition}

\newtheorem{case[theorem]}{Case}

\newcommand{\R}{{\mathbb R}}

\newcommand{\Z}{{\mathbb Z}}

\newcommand{\T}{{\mathbb T}}

\renewcommand{\P}{{\mathbb P}}

\newcommand{\supp}{\operatorname{supp}}

\theoremstyle{remark}
\newtheorem{remark}[theorem]{Remark}

\numberwithin{equation}{section}

\def\R{\Bbb R}

\begin{document}

\title{Uniform distribution and geometric incidence theory}

\author{A. Gafni, A. Iosevich, and E. Wyman}

\date{\today}

\address{Department of Mathematics, University of Mississippi, University, MS}
\email{argafni@olemiss.edu}

\address{Department of Mathematics, University of Rochester, Rochester, NY}
\email{iosevich@math.rochester.edu}

\address{Department of Mathematics, University of Rochester, Rochester, NY}
\email{emmett.wyman@rochester.edu}

\thanks{The second listed author is supported in part by the National Science Foundation grant no. HDR TRIPODS - 1934962}

\maketitle

\begin{abstract} A celebrated unit distance conjecture due to Erd\H os says that that the unit distances cannot arise more than $C_{\epsilon}n^{1+\epsilon}$ times (for any $\epsilon>0$) among $n$ points in the Euclidean plane (see e.g. \cite{SST84} and the references contained therein). In three dimensions, the conjectured bound is $Cn^{\frac{4}{3}}$ (see e.g. \cite{KMSS12} and \cite{Z19}). In dimensions four and higher, this problem, in its general formulation, loses meaning because the Lens example shows that one can construct a set of $n$ points in dimension $4$ and higher where the unit distance arises $\approx n^2$ times (see e.g. \cite{B97}). However, the Lens example is one-dimension in nature, which raises the possibility that the unit distance conjecture is still quite interesting in higher dimensions under additional structural assumptions on the point set. This point of view was explored in \cite{I19}, \cite{IS16}, \cite{IMT12}, \cite{IRU14}, \cite{OO15} and has led to some interesting connections between the unit distance problem and its continuous counterparts, especially the Falconer distance conjecture (\cite{Falc85}). 

In this paper, we study the unit distance problem and its variants under the assumption that the underlying family of point sets is uniformly distributed. We prove several incidence bounds in this setting and clarify some key properties of uniformly distributed sequences in the context of incidence problems in combinatorial geometry. 

\end{abstract} 


\section{Introduction}

The theory of uniform distribution has a long and distinguished history. Recall that a sequence of points $\{v_n\}$ in the unit cube ${[0,1]}^d$, $d \ge 1$, is said to be \emph{uniformly distributed} if for every continuous function $f$, 
\begin{equation} \label{unifregular} 
	\lim_{N \to \infty} \frac{1}{N} \sum_{n=1}^N f(v_n)=\int_{{[0,1]}^d} f(x) dx.
\end{equation} 

 $f(x)=e^{2 \pi i x \cdot k}$, where $k=(k_1, \dots, k_d)$ is a non-zero integer lattice point and  plugging this function into (\ref{unifregular}) yields 
$$ 
	\lim_{N \to \infty} \frac{1}{N} \sum_{n=1}^N e^{2 \pi i v_n \cdot k}=\int_{{[0,1]}^d} e^{2 \pi i x \cdot k} dx=0.
$$
The classical Weyl criterion says that the converse is also true, namely $\{v_n\}$ is uniformly distributed in the unit cube if and only if 
$$ \sum_{n=1}^N e^{2 \pi i v_n \cdot k}=o(N)$$ for every non-zero $k \in {\Bbb Z}^d$. 

The purpose of this paper is to study various aspects of uniform distribution in the context of incidence theorems in combinatorial geometry. One of the central results in geometric combinatorics is the celebrated Szemeredi-Trotter incidence theorem (\cite{ST83}), which bounds the number of incidences between $n$ points and $m$ lines in the Euclidean plane.  More precisely, let $P$ be a set of $n$ points in ${\Bbb R}^2$ and $L$ be a set of $m$ lines. Then 
$$ I(P,L):=\# \{(p,l) \in P \times L: p \in l \} \le C(n+m+{(nm)}^{\frac{2}{3}})$$ 
Moreover, this bound is, in general, best possible. 

\vskip.125in 

The lines can be replaced by circles, and, more generally, other families of geometric objects satisfying certain intersection axioms (see e.g. \cite{SST84}). In the case of circles, there is a related, celebrated unit distance conjecture due to Erd\H os, which says that, with $P$ as above, 
\begin{equation} \label{unitdistancebound} \# \{(p,p') \in P \times P: |p-p'|=1 \} \lessapprox n, \end{equation} where here, and throughout, $X \lessapprox Y$ with the controlling parameter $n$ means that for every $\epsilon>0$ there exists $C_{\epsilon}>0$ such that $X \leq C_{\epsilon}n^{\epsilon} Y$. It is not difficult to see that (\ref{unitdistancebound}) can be viewed as a bound on the number of incidences between the points of $P$ and the unit circles centered at the points of $P$. The best known result in this direction, due to Spencer, Szemeredi and Trotter says that the left hand side in (\ref{unitdistancebound}) is bounded by $Cn^{\frac{4}{3}}$. 

The unit distance problem, in the form stated above, is also quite interesting in three dimensions. The best known result, due to Zahl, says that the number of unit distances determined by $n$ points in ${\Bbb R}^3$ is $\leq Cn^{\frac{295}{197}}$ (\cite{Z19}), beating the previous bound $Cn^{\frac{3}{2}}$ (\cite{KMSS12}). It is interesting to note that in three dimensions, the left hand side of (\ref{unitdistancebound}) is the number of incidences between points of a finite set $P$ and the unit spheres centered at those points. In dimension four and higher, however, the unrestricted version of the unit distance problem becomes meaningless due to the following beautiful example due to Lenz (see e.g. \cite{B97} for a thorough description; see also a nice description in \cite{S08}). Let $P$ consist of $\frac{n}{2}$ equally spaced points on the circle 
$$ \{(\cos(\theta), \sin(\theta), 0, 0): 0 \leq \theta \leq 2 \pi\}$$ and let $P'$ be a set of $\frac{n}{2}$ equally spaced points on the circle 
$$ \{(0,0,\cos(\phi), \sin(\phi)): 0 \leq \phi \leq 2 \pi \}.$$ 

Since all the distances between points of $P$ and $P'$ are equal to $\sqrt{2}$, it is not difficult to see that in dimensions four and higher, the best general bound for the left hand side of (\ref{unitdistancebound}) is the trivial bound $Cn^2$, which makes it clear that without additional assumption, there is no interesting theory in dimensions four and higher. 

However, the set of points in the Lenz example is rather one-dimensional, which suggests that non-trivial results may be obtained under additional structural assumptions on the underlying point set $P$. This point of view has been studied by the second listed author of this paper in a variety of settings (see e.g. \cite{I19}, \cite{IS16}), using the notion of discrete energy. Given a family of finite point sets $\{P_n\}$ in ${\Bbb R}^d$, $d \ge 2$, $n^{-\frac{1}{s}}$-separated for some $s \in (0,d)$ and containing $n$ points each, define the discrete energy of $P_n$ by 
$$ {\mathcal I}_s(P_n)=\frac{1}{n^2} \sum_{p \not=p' \in P_n} {|p-p'|}^{-s}.$$ 
One can show (see e.g. \cite{I19}, \cite{IS16}) that if ${\mathcal I}_s(P_n)$ is bounded above with constants independent of $n$, then, the number of incidences between $n$ points and $n$ annuli of radius $\approx 1$ and thickness $\approx n^{-\frac{1}{s}}$ is bounded by $Cn^{2-\frac{1}{s}}$ if $s>\frac{d+1}{2}$. 

Another natural structural assumption that would avoid the Lenz example above is uniform distribution because, in particular, a uniformly distributed set of points cannot be concentrated on a lower dimensional algebraic variety. 

\begin{definition}\label{def gamma-uniform} For $\gamma \in \left(0, \frac{1}{2} \right]$, we say that a sequence of points $\{v_n\}$ is  $\gamma$-uniformly distributed in ${[0,1]}^d$,  if for every $\epsilon > 0$, there exists a uniform constant $C_\epsilon$ such that for every $k \in {\Bbb Z}^d \setminus 0$, 
$$
	\left| \frac{1}{N} \sum_{n=1}^N e^{2 \pi i k \cdot v_n} \right| \leq C_\epsilon |k|^\epsilon N^{-\gamma}.
$$
\end{definition}

We discuss the validity of Definition \ref{def gamma-uniform} in Section \ref{def discussion}, where we show that random sequences are generically $\gamma$-uniformly distributed for all $\gamma < 1/2$. We also discuss why the factor $|k|^\epsilon$ is necessary in the bound on the right.

\vskip.125in 

To state our results, we require some notation.

\bigskip

\noindent
\textbf{Notation.} Given $x \in \R^d$, we write
\[
	\|x\| = \min_{m \in \Z^d} |x - m|,
\]
the distance $x$ is to the integer lattice $\Z^d$. Note, $\| \cdot \|$ is well-defined on the torus $\T^d$. In particular, $\|x - y\|$ is precisely the distance between points $x$ and $y$ in $\T^d$ as measured by the Riemannian metric. Given a subset $S$ of $\R^d$ or $\T^d$, we write $|S|$ to mean the Lebesgue measure of $S$.

\bigskip

Our main results are the following.

\vskip.125in 

\begin{theorem} \label{sphereonsteroidsth} Suppose that $\{v_n\}$ is a $\gamma$-uniformly distributed sequence of points on 
${\Bbb T}^d$, 
$d \ge 2$. Let 
$$\Omega_{a,b} = \{(x,y) \in \T^d \times \T^d : a \leq \|x - y\| \leq b \},$$ with $\frac{1}{100} \leq a \leq b<\frac{1}{2}$. Then 
\begin{equation} \label{udstest} 
	\# \left\{(n,m) \in {[1,N]}^2: a \leq \| v_n-v_m \| \leq b \right\} = N^2 |\Omega_{a,b}| + R( a,b, N), \end{equation}
where
$$
	|R(a, b, N)| \leq C_\epsilon N^{2 -\frac{4\gamma}{d+1} + \epsilon}.
$$

In particular, if $ b-a \ge cN^{-\frac{4\gamma}{d+1} + \epsilon}$ for some positive constant $c > 0$ and $\epsilon > 0$, then the number of incidences between the first $N$ points of $\{v_n\}$ and $N$ annuli centered at those points is comparable (above and below) to $N^2 (b-a)$ for sufficiently large $N$ since $|\Omega_{a,b}| \approx  b-a.$
\end{theorem}

This result gives an upper bound for point-sphere incidences for $\frac12$-uniformly distributed sequences.

\begin{corollary} Let $\{v_n\}$ be a $\frac12$-uniformly distributed sequence of points in $\T^d$ and fix a distance $0 < t < 1/2$. Then,
\[
	\# \left\{(n,m) \in {[1,N]}^2 : \|v_n - v_m\| = t \right\} \leq C_\epsilon N^{2 - \frac{2}{d+1} + \epsilon}.
\]
\end{corollary}

\vskip.125in 

Our next result deals with thickened hyperplanes instead of annuli, which, interestingly, leads to a considerably different numerology. 

\begin{theorem} \label{dotproductonsteroidsth} Suppose that $\{v_n \}$ and $\{w_m\}$ are $\gamma$-uniformly distributed sequences of points on ${[0,1]}^d$, 
$d \ge 2$. Let $\Omega_{a,b} = \{ (x,y) \in [0,1]^d \times [0,1]^d : a \leq x \cdot y \leq  b\}$, $a \ge \frac{1}{100}$. Let $\psi$ be a smooth, non-negative function with support contained in the open cube $(0,1)^d$. Then 
\begin{equation} \label{uddot} 
	\sum_{\left\{(n,m) \in {[1,N]}^2: a \leq v_n \cdot w_m \leq b \right\}} \psi(v_n)\psi(w_m)= N^2 \int \int_{\Omega_{a,b}} \psi(x) \psi(y) dxdy+ R( a,b, N), \end{equation}
where for any $\epsilon>0$, there exists $C_{\epsilon}>0$ such that 
$$
	|R( a,b, N)| \leq C_{\epsilon}N^{2-\gamma+\epsilon}.
$$

In particular, if $ b-a \ge cN^{-\gamma + \epsilon}$, fixed constant $c > 0$ and $\epsilon > 0$, then the number of incidences between $N$ points of $\{w_m\}$ and $N$ thin slabs with normals $\{v_n\}$ is comparable to (above and below) $N^2 (b-a)$ for sufficiently large $N$ since $|\Omega_{a,b}| \approx b-a$.
\end{theorem}

\begin{corollary}
	Suppose $\{v_n\}$ and $\{w_m\}$ are both $\frac12$-uniformly distributed sequences in $[0,1]^d$, $d \geq 2$, and fix $0 < t < d$. Then for any $\epsilon>0$ there exists $C_{\epsilon}>0$ such that 
\[
	\# \left\{(n,m) \in {[1,N]}^2 : v_n \cdot w_m = t \right\} \leq C_\epsilon N^{\frac32 + \epsilon}. 
\]

\end{corollary}

\vskip.125in 

\begin{remark} 
For $d=2,3$, Theorem \ref{sphereonsteroidsth} gives a bound on the restricted unit distance problem of $N^{\frac{4}{3} + \epsilon}$ and $N^{\frac{3}{2} + \epsilon}$, respectively, which are very close to matching what is known in the unrestricted case. Moreover, our results involve incidences between thickened spheres, which are generally not covered by combinatorial techniques. \end{remark} 

\begin{remark} When $\gamma= \frac{1}{2}$, the bound in Theorem \ref{sphereonsteroidsth} matches the bound from \cite{I19} and  \cite{IS16} (see also \cite{OO15}) for point sets with bounded discrete energy, by setting $b = N^{-\frac{1}{2}}$.  This suggests a deeper connection between uniform distribution and discrete energy. We shall address this point in a sequel. \end{remark} 

\begin{remark} In Theorem \ref{generically 1/2} below we show that an i.i.d.~sequence of random variables is $\gamma$-uniformly distributed for all $\gamma<\frac{1}{2}$ with probability $1$. Consider $N$ randomly chosen points in the unit cube ${[0,1]}^d$, $d \ge 2$. Then the expected number of pairs of points separated by at least $1-\epsilon$ and at most $1+\epsilon$ is approximately equal to $N$ times the volume of a sphere of radius $1$ and thickness $\epsilon$, which is approximately $N \epsilon$. Thus the expected number of pairs is $\approx N^2 \epsilon$. The same argument works for spheres replaced by planes. This shows that the $N^2 \epsilon$ term in the results above is correct, though the nature of the error terms is much more subtle. \end{remark}

\subsection{The general framework of the proofs} \label{FRAMEWORK}

\vskip.125in 

We will identify the unit cube $[0,1]^d$ with the torus $\T^d = \R^d/\Z^d$. Let $\Omega \subset \T^d \times \T^d$ be some nice relation on $\T^d$ (e.g. positive measure in $\T^d \times \T^d$, smooth boundary). The goal is, given a sequence $\{v_n\}_{n = 1}^\infty$ in $\T^d$, to estimate
\[
	\#\{ (n,m) : n,m \in \{1,\ldots,N\}, \ (v_n, v_m) \in \Omega\}
\]
quantitatively.

This problem can be viewed as a twist on the classical Weyl's criterion, for which a strategy for estimating such quantities was outlined by Colzani, Gigante, and Travaglini \cite{CGT11}. Our approach to Theorems \ref{sphereonsteroidsth} and \ref{dotproductonsteroidsth} roughly follows theirs, but with some careful decisions about how certain estimates are made.

The rough strategy goes as follows. We construct a smooth cutoff $\chi_\delta$ which differs from the indicator function of $\Omega$ only in an $\delta$-neighborhood of the boundary. We can estimate the count above by $\sum_{n,m = 1}^N \chi_\delta(v_n,v_m)$ with discrepancy
\[
	\sum_{n,m = 1}^N |(\mathbf 1_\Omega - \chi_\delta)(v_n,v_m)| \leq \sum_{n,m = 1}^N \tilde \chi_\delta (v_n,v_m),
\]
where here $\tilde \chi_\delta$ is a smooth function supported on an $\delta$-neighborhood of the boundary of $\Omega$ so that $\tilde \chi_\delta \geq |\mathbf 1_\Omega - \chi_\delta|$. We then estimate both
\[
	\sum_{n,m = 1}^N \chi_\delta(v_n,v_m) \qquad \text{ and } \qquad \sum_{n,m = 1}^N \tilde \chi_\delta(v_n,v_m)
\]
using the Fourier series. The main term for the count will be the contribution of the zeroth Fourier coefficient of $\chi_\delta$ to the first sum, which is typically equal to the measure $|\Omega|$ by design. The remainder term is bounded by
\[
	\left| \sum_{k \neq 0} \sum_{n,m = 1}^N \widehat \chi_\delta(k) e^{2\pi i k \cdot (v_n,v_m)} \right| + \left| \sum_{k \in \Z^{2d}} \sum_{n,m = 1}^N \widehat{ \tilde \chi}_\delta(k) e^{2\pi i k \cdot (v_n,v_m)} \right|,
\]
and by the triangle inequality also by
\begin{equation} \label{remainder reduction}
	N^2 \iint \tilde \chi_\delta(x,y) \, dx \, dy + \sum_{k \neq 0} (|\widehat \chi_\delta(k)| + |\widehat {\tilde \chi}_\delta(k)|) \left| \sum_{n,m = 1}^N  e^{2\pi i k \cdot (v_n,v_m)} \right|.
\end{equation}

Estimating the remainder then becomes a game of choosing $\delta$ which optimizes the sum against the integral. The $\gamma$-uniform distribution property helps us bound the sum in $n$ and $m$ in the second term, while the regularity of the boundary of $\Omega$ usually allows us to bound the first term by $N^2 \delta$.

\vskip.25in


\section{Discussion of the definition of a $\gamma$-uniformly distributed sequence} \label{def discussion}

Definition \ref{def gamma-uniform} is a quantitative analog of the usual notion of uniform distribution, up to a factor which is independent of $N$.  The restriction of $\gamma\le \frac12$ should seem natural to number theorists who are familiar with the philosophy of \emph{square root cancellation}.  The square root barrier is an extension of a basic fact from probability:  For a simple random walk in $\mathbb Z$  (i.e., the sum of a uniform random sequence in $\{-1,1\}$), the expected distance from the origin after $N$ steps is on the order of $\sqrt N$.  Thus we cannot hope that Definition \ref{def gamma-uniform} would hold for a uniform random sequence when $\gamma > \frac12$.  In fact, as we will see below, it is not possible for any sequence to be $\gamma$-uniformly distributed with $\gamma > \frac12$.  The threshold case $\gamma = \frac12$ characterizes sequences that most closely resemble random sequences, at least with respect to this ``random walk'' statistic.  This fact is demonstrated in the following proposition.

\begin{proposition}\label{generically 1/2}
A sequence of i.i.d. random variables drawn with uniform probability from the cube $[0,1]^d$ is, with probability $1$, $\gamma$-uniformly distributed for all $\gamma < 1/2$.
\end{proposition}

\begin{proof}
Let $v_1,v_2,\ldots$ be the sequence in the proposition. It suffices to show this sequence is, for any fixed $\gamma < 1/2$, almost certainly $\gamma$-uniformly distributed. Note for $k \neq 0$, $e^{2\pi i k \cdot v_n}$ are also i.i.d. random variables with mean $0$ and have real and imaginary parts ranging between $[-1,1]$. Applying Hoeffding's inequality to both real and imaginary parts and using the union bound yields
\[
	\P\left[\left| \frac1N \sum_{n = 1}^N e^{2\pi i k \cdot v_n} \right| \geq |k|^\epsilon N^{- \gamma} \right] \leq 4\exp\left(-\frac12|k|^{2\epsilon} N^{1-2\gamma}\right).
\]
Furthermore, the probability above is nonzero only when $|k| \leq N^{\gamma/\epsilon}$. Hence, another union bound yields
\begin{align*}
	\P\left[\sup_{k \neq 0} \left| |k|^{-\epsilon} N^{\gamma - 1} \sum_{n = 1}^N e^{2\pi i k \cdot v_n} \right| \geq 1 \right] &\leq \sum_{0 < |k| \leq N^{\gamma/\epsilon}} 4\exp\left(-\frac12|k|^{2\epsilon} N^{1-2\gamma}\right) \\
	&\leq C N^{d\gamma/\epsilon} \exp\left(-\frac12 N^{1-2\gamma}\right).
\end{align*}
We fix $N_0 \geq 1$ and again bound
\begin{align*}
	\P\left[\sup_{N \geq N_0} \sup_{k \neq 0} \left| |k|^{-\epsilon} N^{\gamma - 1} \sum_{n = 1}^N e^{2\pi i k \cdot v_n} \right| \geq 1 \right] &\leq \sum_{N \geq N_0} C N^{d\gamma/\epsilon} \exp\left(-\frac12 N^{1-2\gamma}\right),
\end{align*}
which vanishes as $N_0 \to \infty$. We conclude
\[
	\limsup_{N \to \infty} \sup_{k \neq 0} |k|^{-\epsilon} N^\gamma \left| \frac{1}{N} \sum_{n = 1}^N e^{2\pi i k \cdot v_n} \right| \leq 1
\]
almost certainly. The proposition follows.
\end{proof}

The $|k|^\epsilon$ growth factor in Definition \ref{def gamma-uniform} is necessary in the sense that, if it were not present, \emph{no} sequence would be $\gamma$-uniformly distributed. To illustrate, we consider any sequence in $[0,1]$. Taking only the first $N$ points $v_1,\ldots, v_N$, there exists a possibly large frequency $k$ at which
\[
	\frac1N \sum_{n = 1}^N e^{2\pi i k v_n}
\]
is close to $1$. This is clear if $v_1,v_2,\ldots$ is a sequence of rational numbers, but it also holds in general as follows. For each $\epsilon > 0$, we use Dirichlet simultaneous approximation to produce a denominator $q \leq \epsilon^{-N}$ and numerators $p_1,\ldots,p_N$ for which
\[
	|q v_n - p_n| \leq \epsilon \qquad \text{ for } 1 \leq n \leq N.
\]
Hence,
\[
	\left| \frac1N \sum_{n = 1}^N e^{2\pi i q v_n} - 1 \right| \leq \frac1N \sum_{n = 1}^N \left| e^{2\pi i (q v_n - p_n)} - 1 \right| \leq \frac{1}{N} \sum_{n = 1}^N 2\pi \epsilon \leq 2\pi \epsilon.
\]

Intuitively, we do not expect a $\gamma$-uniformly distributed sequence to have much repetition. We quantify this here.

\begin{proposition}\label{support proposition}
	Let $\{v_n\}_{n = 1}^\infty$ be a $\gamma$-uniformly distributed sequence in $[0,1]^d$. Then for every $\epsilon > 0$, there exists a constant $c_\epsilon > 0$ for which 
	\[
		|\{v_1,\ldots,v_N\}| \geq c_\epsilon N^{2\gamma - \epsilon} \qquad N \geq 1.
	\]
\end{proposition}

We remark that in the maximal case $\gamma = 1/2$, then the size of $\{v_1,\ldots,v_N\}$ is nearly comparable to $N$. This observation has two important consequences. First:

\begin{corollary} There are no $\gamma$-uniformly distributed sequences for $\gamma > 1/2$.
\end{corollary}

Secondly, $1/2$-uniformly distributed sequences have little to no additive structure.

\begin{corollary}\label{no arithmetic structure}
	Let $\{v_n\}_{n = 1}^\infty$ be a $1/2$-uniformly distributed sequence in the torus $\T^d$. Then, for every $\epsilon > 0$, there exists a constant $c_\epsilon > 0$ for which 
	\[
		|\{v_n - v_m : n,m \in \{1,\ldots,N\}\}| \geq c_\epsilon N^{2 - \epsilon} \qquad N \geq 1.
	\]
\end{corollary}

We first prove Proposition \ref{support proposition} and then prove Corollary \ref{no arithmetic structure} which, admittedly, is a corollary of the proof of the proposition rather than of the proposition itself.

\begin{proof}[Proof of Proposition \ref{support proposition}]
We make a couple convenient reductions. First, we identify $[0,1]^d$ with the flat torus $\T^d$. Secondly, we prove a slightly more general statement: Given a sequence of probability measures $\mu_1,\mu_2,\ldots$ on $\T^d$, each of finite support, satisfying the $\gamma$-uniform distributivity bounds
\begin{equation}\label{ud condition for measures}
	|\widehat \mu_m(k)| \leq C_\epsilon |k|^\epsilon N_m^{-\gamma} \qquad k \neq 0, \epsilon > 0,
\end{equation}
we have
\begin{equation}\label{general support proposition}
	\# \supp \mu_m \geq c_\epsilon N_m^{2\gamma - \epsilon} \qquad \epsilon > 0.
\end{equation}
The proposition follows after taking $\mu_N$ to be the partial averaging measure
\[
	\mu_N = \frac1N \sum_{n = 1}^N \delta_{v_n}.
\]

Let $B_\delta$ denote both the distance ball of radius $\delta$ centered at $0$ and its indicator function, as context indicates. We have
\[
	|\supp \mu_m * B_\delta| \leq \left| \bigcup_{x \in \supp \mu_m} (x + B_\delta) \right| \leq \sum_{x \in \supp \mu_m} |B_\delta| = \# \supp \mu_m |B_\delta|,
\]
and hence
\[
	\# \supp \mu_m \gtrsim \delta^{-d} |\supp \mu_m * B_\delta|.
\]
Using the Cauchy-Schwarz inequality,
\[
	\# \supp \mu_m \geq \frac{\left( \int \mu_m * B_\delta \right)^2}{\epsilon^d \int (\mu_m * B_\delta)^2}.
\]
The numerator is exactly $|B_\delta|^2 \gtrsim \delta^{2d}$. For the integral in the denominator, Plancherel and \eqref{ud condition for measures} yields
\begin{align*}
	\int (\mu_m * B_\delta)^2 &= \sum_{k \in \Z^d} |\widehat \mu_m(k)|^2 |\widehat B_\delta(k)|^2 \\
	&\lesssim \delta^{2d} + C_\epsilon \delta^{2d} N_m^{-2\gamma} \sum_{k \neq 0} |k|^\epsilon |\widehat B_1(\delta k)|^2.
\end{align*}
The standard stationary phase bounds yield $|\widehat B_1(\delta k)|^2 \lesssim \min(1, |\delta k|^{-d-1})$, so the above is
\begin{align*}
	\lesssim \delta^{2d} + C_\epsilon \delta^{2d} N_m^{-2\gamma} \left( \delta^{-d-\epsilon} + \delta^{-d-1} \sum_{|k| \geq \delta^{-1}} |k|^{-d-1 + \epsilon} \right) \lesssim \delta^{2d} + C_\epsilon \delta^{d - \epsilon} N_m^{-2\gamma}.
\end{align*}
To summarize, we have
\[
	\# \supp \mu_m \gtrsim \frac{1}{\delta^d + C_\epsilon \delta^{-\epsilon} N_m^{-2\gamma}}.
\]
Optimizing with $\delta = N_m^{-2\gamma/d}$ yields \eqref{general support proposition}.
\end{proof}

\begin{proof}[Proof of Corollary \ref{no arithmetic structure}]
Let
\[
	\nu_m = \frac{1}{m} \sum_{i = 1}^m \delta_{x_i}
\]
and note
\[
	|\widehat \nu_m(k)| \leq C_\epsilon |k|^\epsilon m^{-1/2} \qquad k \neq 0, \epsilon > 0
\]
by hypothesis. Let
\[
	\mu_m = \nu_m * \nu_m = \frac{1}{m^2} \sum_{i,j = 1}^m \delta_{x_i - x_j}.
\]
Note for $k \neq 0$, we have
\[
	|\widehat \mu_m(k)| = |\widehat \nu_m(k)|^2 \leq C_\epsilon |k|^{2\epsilon} m^{-1},
\]
and hence $\mu_m$ satisfies the uniform distribution condition \eqref{ud condition for measures} for $\gamma = 1/2$ and $N_m = m^2$. Hence by \eqref{general support proposition},
\[
	\#\supp \mu_m \geq c_\epsilon N^{2 - \epsilon} \qquad \epsilon > 0,
\]
as needed.
\end{proof}

\vskip.25in 


\section{Proof of Theorem \ref{sphereonsteroidsth}} \label{SPHERES}

\vskip.125in 

	Let $\rho$ be a nonnegative smooth function on $\R^d$ with support in the unit ball such that $\rho(x) \geq c > 0$ for $|x| \leq \delta/2$ and $\int \rho = 1$. Let $\rho_\delta$ be defined on the torus by taking
	\[
		\rho_\delta(x) = \delta^{-d} \sum_{k \in \Z^d} \rho(\delta^{-1} (x - k)) = \sum_{k \in \Z^d} e^{2\pi i x \cdot k} \widehat \rho(\delta k),
	\]
	with the second equality following by Poisson summation. Note, $\int \rho_\delta = \widehat \rho_\delta(0) = 1$ and $\supp \rho_\delta$ is contained in a metric ball of radius $\delta$ about $0$ in $\T^d$. 
	
	Let $A_{[a,b]}$ denote the annulus $\{x \in \T^d : a \leq \|x\| \leq b\}$ and set
	\[
		\chi_\delta = \mathbf 1_{A_{[a,b]}} * \rho_\delta.
	\]
	We must also select an appropriate smooth $\tilde \chi_\delta$ to bound the discrepancy
	\[
		|\chi_\delta - \mathbf 1_{A_{[a,b]}}| \leq \tilde \chi_\delta.
	\]
	To this end, we note
	\[
		|\chi_\delta(x) - \mathbf 1_{A_{[a,b]}}(x)| \leq 
		\begin{cases}
			0 & \text{ if } d_{\T^n}(\partial A_{[a,b]}, x) > \delta \\
			1 & \text{ if } d_{\T^n}(\partial A_{[a,b]}, x) \leq \delta
		\end{cases}
		\leq 2 c^{-1} \delta (\rho_{2\delta} * \sigma_a + \rho_{2\delta} * \sigma_b)
	\]
	where $\sigma_r$ denotes the surface measure on the sphere of radius $r$, and where the second inequality holds for $\delta$ smaller than some threshold depending on $a$. So, we set
	\[
		\tilde \chi_\delta = 2 c^{-1} \delta (\rho_{2\delta} * \sigma_a + \rho_{2\delta} * \sigma_b).
	\]
	We have
	\begin{equation}\label{spheres eq1}
		\left| N^{-2} \sum_{n,m = 1}^N \mathbf 1_{A_{[a,b]}}(v_n - v_m) - N^{-2} \sum_{n,m = 1}^N \chi_\delta (v_n - v_m) \right| \leq N^{-2} \sum_{n,m = 1}^N \tilde \chi_\delta(v_n - v_m).
	\end{equation}
	We turn our attention to the second term in absolute values, which we write as
	\begin{align*}
		N^{-2} \sum_{n,m = 1}^N \chi_\delta (v_n - v_m) &= N^{-2} \sum_{k \in \Z^d} \sum_{n,m=1}^N \widehat \chi_\delta(k) e^{2\pi i (v_n - v_m) \cdot k} \\
		&= |A_{[a,b]}| + \sum_{k \neq 0} \widehat \chi_\delta(k) \left( N^{-2} \sum_{n,m=1}^N e^{2\pi i (v_n - v_m) \cdot k} \right) \\
		&= |\Omega_{a,b}| + \sum_{k \neq 0} \widehat \chi_\delta(k) \left| N^{-1} \sum_{n=1}^N e^{2\pi i v_n \cdot k} \right|^2.
	\end{align*}
	We similarly write the right side of \eqref{spheres eq1} as
	\[
		N^{-2} \sum_{n,m = 1}^N \tilde \chi_\delta(v_n - v_m) = \widehat{\tilde \chi_\delta}(0) + \sum_{k \neq 0} \widehat{\tilde \chi_\delta}(k) \left| N^{-1} \sum_{n=1}^N e^{2\pi i v_n \cdot k} \right|^2.
	\]
	Turning again to \eqref{spheres eq1}, we find
	\begin{multline}\label{spheres eq2}
		\left| N^{-2} \sum_{n,m = 1}^N \mathbf 1_{A_{[a,b]}}(v_n - v_m) - |\Omega_{a,b}| \right| \\ \leq \widehat{\tilde \chi_\delta}(0) + \sum_{k \neq 0} \left( |\widehat \chi_\delta(k)| + |\widehat{\tilde \chi_\delta}(k)| \right) \left| N^{-1} \sum_{n=1}^N e^{2\pi i v_n \cdot k} \right|^2.
	\end{multline}
	We now must estimate the Fourier coefficients of $\chi_\delta$ and $\tilde \chi_\delta$ and invoke the $\gamma$-uniform distribution hypotheses to bound the rightmost term. First we establish that
	\begin{equation}\label{sphere zeroth coefficient}
		\widehat{\tilde \chi_\delta}(0) = 2c^{-1} \delta \int (\rho_{2\delta} * \sigma_a(x) + \rho_{2\delta} * \sigma_b(x)) \, dx \leq C_d c^{-1} \delta.
	\end{equation}
	Next, we bound $|\widehat \chi_\delta|$ and $|\widehat {\tilde \chi_\delta}|$. To this end, we have a standard lemma.
	
	\begin{lemma} (See Lemma 2.1 and its proof in \cite{Falc85}) There exists a constant $C$ such that for $\frac{1}{100} \leq a,b,r \leq \frac{1}{2}$ and $a < b$, 
		\[
			|\widehat \sigma_r(\xi)| \leq C |\xi|^{-\frac{d-1}{2}} \qquad \text{ and } \qquad |\widehat{\mathbf 1}_{A_{[a,b]}}(\xi)| \leq C {|\xi|}^{-\frac{d+1}{2}}.
		\]
	\end{lemma}
	
In light of the lemma and the $\gamma$-uniform distribution hypothesis, we bound the right side of \eqref{spheres eq2} by a constant multiple of
\[
	\delta + N^{-2\gamma} \sum_{k \neq 0} |\widehat \rho(\delta |k|)| |k|^{-\frac{d+1}{2} + 2\epsilon} + \delta N^{-2\gamma} \sum_{k \neq 0} |\widehat \rho(2\delta |k|)| |k|^{-\frac{d-1}{2} + 2\epsilon}.
\]
By the rapid decay of $\widehat \rho$ and routine estimates, the second and third terms are both bounded by $N^{-2\gamma} \delta^{-\frac{d-1}{2} - 2\epsilon}$, and we have that \eqref{spheres eq2} is bounded by
\[
	\delta + N^{-2\gamma} \delta^{-\frac{d-1}{2} - 2\epsilon}.
\]
Optimizing at $\delta = N^{-\frac{4\gamma}{d+1}}$ yields the theorem.

\vskip.25in 


\section{Proof of Theorem \ref{dotproductonsteroidsth}} 

The theorem will follow from the next lemma, for which we require a smooth, nonnegative function $\rho$ on $\R$ with $\int \rho(t) \, dt = 1$ and with $\supp \rho \subset [-1,1]$. We also take the scaling
\[
	\rho_\delta(t) = \delta^{-1} \rho(\delta^{-1} t)
\]
so that $\rho_\delta$ is an approximate identity as $\delta \to 0$.

\begin{lemma}\label{dotproductonsteroidslem}
	Suppose that $\{v_n\}$ and $\{w_m\}$ are $\gamma$-uniformly distributed sequences of points on $[0,1]^d$. Let $\rho_\delta$ be as above, and furthermore let $\psi$ be a smooth, non-negative function with support contained in the open cube $(0,1)^d$. Then,
	\begin{equation}\label{dotlem1}
		N^{-2} \sum_{n = 1}^N \sum_{m = 1}^N \rho_\delta(v_n \cdot w_m - t) \psi(v_n) \psi(w_m) = \iint \rho_\delta(x \cdot y - t) \psi(x) \psi(y) \, dx \, dy + O_\epsilon(\delta^{-1-\epsilon} N^{-\gamma}),
	\end{equation}
	where the constants implicit in the big-$O$ notation are independent of $t$ with $\frac{1}{100} \leq t \leq \frac{1}{2}$. Furthermore,
	\begin{equation}\label{dotlem2}
		N^{-2} \sum_{n = 1}^N \sum_{m = 1}^N \mathbf 1_{[a,b]} * \rho_\delta(v_n \cdot w_m) \psi(v_n) \psi(w_m) =  \iint \mathbf 1_{[a,b]} * \rho_\delta(x \cdot y) \psi(x) \psi(y) \, dx \, dy + O_\epsilon(\delta^{-\epsilon} N^{-\gamma}).
	\end{equation}
\end{lemma}

\vskip.125in 

\begin{remark} Note, the dot product $x \cdot y$ is ambiguously defined for $x$ and $y$ on the torus, but identifying $\T^d$ with the cube $[0,1]^d$ fixes an interpretation. The cutoff $\psi$ is then used to deal with the complications arising at the boundary of $[0,1]^d$ where the torus has been cut. 
\end{remark} 

\vskip.125in 

\begin{proof}
	We claim that for $y$ fixed we have
	\begin{equation}\label{dotproductonsteroidslem claim}
		N^{-1} \sum_{n = 1}^N \rho_\delta(v_n \cdot y - t) \psi(v_n) = \int_{[0,1]^d} \rho_\delta(x \cdot y - t) \psi(x) \, dx + O_\epsilon(\delta^{-1-\epsilon} N^{-\gamma}).
	\end{equation}
	This claim applied twice yields
	\begin{align*}
		&N^{-2} \sum_{n = 1}^N \sum_{m = 1}^N \rho_\delta(v_n \cdot w_m - t) \psi(v_n) \psi(w_m) \\
		&= N^{-1} \sum_{m = 1}^N \int \rho_\delta(x \cdot w_m - t) \psi(x) \psi(w_m) \, dx + O_\epsilon(\delta^{-1 - \epsilon}N^{-\gamma}) \\
		&= \iint \rho_\delta(x \cdot y - t) \psi(x) \psi(y) \, dx \, dy + O_\epsilon(\delta^{-1-\epsilon} N^{-\gamma}),
	\end{align*}
	as needed for \eqref{dotlem1}.
	
	By applying Fourier inversion to $\rho_\delta$, we write the left side of \eqref{dotproductonsteroidslem claim} as
	\[
		\int_{-\infty}^\infty e^{-2\pi i \lambda t} \widehat \rho(\delta \lambda) \left( N^{-1} \sum_{n = 1}^N e^{2\pi i \lambda v_n \cdot y} \psi(v_n) \right) \, d\lambda.
	\]
	The quantity in the parentheses evaluates to
	\begin{align*}
		N^{-1} \sum_{n = 1}^N e^{2\pi i \lambda v_n \cdot y} \psi(v_n) &= N^{-1} \sum_{n = 1}^N \sum_{k \in \Z^d}  e^{2\pi i k \cdot v_n} \widehat \psi(k - \lambda y) \\
		&= \widehat \psi(-\lambda y) + \sum_{k \neq 0} \widehat \psi(k - \lambda y) \left( N^{-1} \sum_{n = 1}^N e^{2\pi i k \cdot v_n} \right).
	\end{align*}
	The rightmost term is $O_\epsilon(\lambda^\epsilon N^{-\gamma})$ by the rapid decay of $\widehat \psi$ and the $\gamma$-uniformly distributed property of the sequence $\{v_n\}$, and hence contributes the $O_\epsilon(\delta^{-1 - \epsilon} N^{-\gamma})$ to \eqref{dotproductonsteroidslem claim}. The contribution of the $k = 0$ term is
	\begin{align*}
		\int_{-\infty}^\infty e^{-2\pi i \lambda t} \widehat \rho(\delta \lambda) \widehat \psi(-\lambda y) \, d\lambda &= \int_{-\infty}^\infty \int e^{2\pi i \lambda (x \cdot y - t)} \widehat \rho(\delta \lambda) \psi(x) \, dx \, d\lambda \\
		&= \int \rho_\delta(x \cdot y - t) \psi(x) \, dx,
	\end{align*}
	which yields the main term. This concludes the proof of the claim \eqref{dotproductonsteroidslem claim}.
	
	\eqref{dotlem2} will similarly follow provided we show, for fixed $y$,
	\[
		N^{-1} \sum_{n = 1}^N \mathbf 1_{[a,b]} * \rho_\delta(v_n \cdot y) \psi(v_n) = \int_{[0,1]^d} \mathbf 1_{[a,b]} * \rho_\delta(x \cdot y) \psi(x) \, dx + O_\epsilon(\delta^{-\epsilon} N^{-\gamma}),
	\]
	which follows by tracing through the proof of \eqref{dotproductonsteroidslem claim} above after replacing $\widehat \rho(\delta \lambda)$ with $\widehat{\mathbf 1}_{[a,b]}(\lambda) \widehat \rho(\delta \lambda)$ and using
	\[
		|\widehat{\mathbf 1}_{[a,b]}(\lambda)| \leq C |\lambda|^{-1}.
	\]
\end{proof}

Next, we use the lemma and the strategy of Section \ref{FRAMEWORK} to estimate the remainder in Theorem \ref{dotproductonsteroidsth}. We write
\[
	|R(a,b,N)| = \left| N^{-2} \sum_{n,m \leq N} \mathbf 1_{\Omega_{a,b}}(v_n, w_m) \psi(v_n) \psi(w_m) - \iint_{\Omega_{a,b}} \psi(x) \psi(y) \, dx \, dy \right| \leq I + II + III
\]
where here
\begin{align*}
	I &= N^{-2} \sum_{n,m \leq N} \left| \mathbf 1_{\Omega_{a,b}}(v_n, w_m) - \int_a^b \rho_\delta(v_n \cdot w_m - t) \, dt \right| \psi(v_n) \psi(w_m),\\
	II &= \iint \left| \mathbf 1_{\Omega_{a,b}}(x,y) - \int_a^b \rho_\delta(x \cdot y - t) \right| \psi(x) \psi(y) \, dx \, dy, \\ 
	III &= \left| N^{-2} \sum_{n,m\leq 1} \int_a^b \rho_\delta(v_n \cdot w_m - t) \, dt \psi(v_n)\psi(w_m) - \iint \int_a^b \rho_\delta(x \cdot y - t) \, dt \psi(x) \psi(y) \, dx \, dy \right|.
\end{align*}
To estimate the terms $I$, $II$, and $III$, we will need a couple estimates. First, we select $\rho$ so that $\rho(t) \geq c > 0$ on $[-1/2,1/2]$. Then, similar to the argument in Section \ref{SPHERES}, we write
\[
	|\mathbf 1_{[a,b]}(t) - \mathbf 1_{[a,b]} * \rho_\delta(t)| \leq 2\delta c^{-1} (\rho_{2\delta}(t - a) + \rho_{2\delta}(t - b)).
\]
From this, we have for each $x,y$,
\begin{equation}\label{planes eq 1}
	\left|\mathbf 1_{\Omega_{a,b}}(x,y) - \int_a^b \rho_\delta(x \cdot y - t) \, dt \right| \leq 2\delta c^{-1} ( \rho_{2\delta}(x \cdot y - a) + \rho_{2\delta}(x \cdot y - b) ).
\end{equation}
To estimate the right side, we will assume $\delta < \frac{1}{400}$ so that for $t \in [a,b]$, $|\nabla_{x,y} (x \cdot y - t)|$ is bounded below by a uniform constant for $x \cdot y - t$ in the support of $\rho_{2\delta}$. It follows
\begin{equation} \label{planes, right side tool}
	\iint \rho_{2\delta}(x \cdot y - t) \psi(x) \psi(y) \, dx \, dy \lesssim 1.
\end{equation}
This immediately yields
\[
	II \lesssim \delta.
\]
\eqref{dotlem2} yields
\[
	III = O_\epsilon(\delta^{-\epsilon} N^{-\gamma}).
\]
By \eqref{planes eq 1}, \eqref{dotlem1}, and \eqref{planes, right side tool}, the remaining term is bounded by
\begin{align*}
	I &\leq 2\delta c^{-1} N^{-2} \sum_{n,m \leq 1} \left( \rho_{2\delta}(v_n \cdot w_m - a) + \rho_{2\delta}(v_n \cdot w_m - b) \right) \psi(v_n) \psi(w_m) \\
	&= 2\delta c^{-1} \iint \left( \rho_{2\delta}(x \cdot y - a) + \rho_{2\delta}(x \cdot y - b) \right) \psi(x) \psi(y) \, dx \, dy + O_\epsilon(\delta^{-\epsilon} N^{-\gamma}) \\
	&\leq C \delta + O_\epsilon(\delta^{-\epsilon} N^{-\gamma}).
\end{align*}
Putting everything together, we have
\[
	|R(a,b,N)| \lesssim \delta + C_\epsilon \delta^{-\epsilon} N^{-\gamma}.
\]
The theorem follows after optimizing at $\delta = N^{-\gamma}$.

\vskip.125in 

\end{document}